%% file: index.tex
\documentclass{article}

\usepackage{microtype}
\usepackage{graphicx}
\usepackage{subfigure}
\usepackage{booktabs} 

\usepackage{hyperref}

\usepackage[accepted]{icml2020}

\input{defs.tex}

\icmltitlerunning{Stochastic Frank-Wolfe for Constrained Finite-Sum Minimization}

\begin{document}




\twocolumn[
    \icmltitle{Stochastic Frank-Wolfe for Constrained Finite-Sum Minimization}

    \icmlsetsymbol{equal}{*}
    
    \begin{icmlauthorlist}
    \icmlauthor{Geoffrey N\'egiar}{ucb}
    \icmlauthor{Gideon Dresdner}{eth}
    \icmlauthor{Alicia Yi-Ting Tsai}{ucb}\\
    \icmlauthor{Laurent El Ghaoui}{ucb,sumup}
    \icmlauthor{Francesco Locatello}{eth,mpi}
    \icmlauthor{Robert M. Freund}{mit}
    \icmlauthor{Fabian Pedregosa}{google}
    \end{icmlauthorlist}

    \icmlaffiliation{ucb}{Berkeley AI Research, University of California, Berkeley, CA, USA}
    \icmlaffiliation{google}{Google Research}
    \icmlaffiliation{eth}{Department of Computer Science, ETH Zurich, Switzerland}
    \icmlaffiliation{mpi}{Max-Planck Institute for Intelligent Systems, T\"ubingen, Germany}
    \icmlaffiliation{sumup}{SumUp Analytics}
    \icmlaffiliation{mit}{MIT Sloan School of Management}
    \icmlcorrespondingauthor{Geoffrey N\'egiar}{\mbox{geoffrey\_negiar@berkeley.edu}}
    
    \vskip 0.3in
]
\printAffiliationsAndNotice{}
\begin{abstract}
We propose a novel Stochastic Frank-Wolfe (a.\/k.\/a.\ conditional gradient) algorithm for constrained smooth finite-sum minimization with a generalized linear prediction/structure. This class of problems includes empirical risk minimization with sparse, low-rank, or other structured constraints.
The proposed method is simple to implement, does not require step-size tuning, and has a constant per-iteration cost that is independent of the dataset size.
Furthermore, as a byproduct of the method we obtain a stochastic estimator of the Frank-Wolfe gap that can be used as a stopping criterion.
Depending on the setting, the proposed method matches or improves on the best computational guarantees for Stochastic Frank-Wolfe algorithms. Benchmarks on several datasets highlight different regimes in which the proposed method exhibits a faster empirical convergence than related methods.
Finally, we provide an implementation of all considered methods in an open-source package.
\end{abstract}

\section{Introduction}

We consider constrained finite-sum optimization problems of the form

\begin{empheq}[box=\mybluebox]{equation*}\tag{OPT}\label{eq:obj_fun}
 \minimize_{\ww \in \CC}\, \frac{1}{n}\sum_{i=1}^n f_i\left(\xx_i\tran \ww\right),
\end{empheq} 

where $\CC$ is a compact and convex set and $\XX = (\xx_1, \cdots, \xx_n)\tran \in \RR^{n \times d}$ is a data matrix, with $n$ samples and $d$ features. This template includes several problems of interest, such as constrained empirical risk minimization. The LASSO \cite{tibshirani96regression} may be written in this form, where $f_i(\xx_i\tran\ww) = \mfrac{1}{2}(\xx_i\tran\ww - y_i)^2$ and $\CC = \{\ww : \|\ww\|_1 \leq \lambda\}$ for some parameter $\lambda$. We focus on the case where the $f_i$s are differentiable with $L$-Lipschitz derivative, and study the convex and non-convex cases. 

\looseness=-1The classical Frank-Wolfe (FW) or Conditional Gradient algorithm \cite{frank1956algorithm, levitin1966constrained, dem1967minimization} is an algorithm for constrained optimization. Contrary to other projection-based constrained optimization algorithms, such as Projected Gradient Descent, it relies on a Linear Minimization Oracle (LMO) over the constraint set $\cal C$, rather than a Quadratic Minimization Oracle (the projection subroutine). 
For certain constraint sets such as the trace norm or most $\ell_p$ balls, the LMO can be computed more efficiently than the projection subroutine. Recently, the Frank-Wolfe algorithm has garnered much attention in the machine learning community where polytope constraints and sparsity are of large interest, e.g. \citet{jaggi2013revisiting, lacoste2015global, locatello17a}. 


\looseness=-1In the unconstrained setting, stochastic variance-reduced methods \cite{shalev2013stochastic, Schmidt2013MinimizingFS, hofmann2015variance} exhibit the same iteration complexity as full gradient (non-stochastic) methods, while reaching much smaller per-iteration complexity, usually at some (small) additional memory cost.
This work takes a step in the direction of designing such a method for Frank-Wolfe type algorithms, which remains an important open problem.

\begin{table}[t]
\caption{Worst-case convergence rates for the function suboptimality after $t$ iterations, for a dataset with $n$ samples. $\kappa \leq n$ and can be much smaller than $n$ for datasets of interest. $\kappa$ is introduced in Section~\ref{sec:discussion}.}
\label{tab:rates}
\begin{center}
\begin{small}
\begin{sc}
\begin{tabular}{lcc}
\toprule
Related Work & Convex & Non-Convex \\
\midrule
\citet{frank1956algorithm} & $\mathcal{O}\left({n}/{t}\right)$ & $\mathcal O\left({n}/\sqrt{t}\right)$  \\
\citet{mokhtari2018stochastic}   & $\mathcal{O}\left(1/\sqrt[3]{t}\right)$ & {\red\large\xmark} \\
\citet{lu2018generalized} & $\mathcal{O}\left({n}/{t}\right)$ & {\red\large\xmark}  \\
This work & $\mathcal{O}\left(\kappa/t\right)$ & $\to 0$ \\
\bottomrule
\end{tabular}
\end{sc}
\end{small}
\end{center}
\vskip -0.2in
\end{table}

\pagebreak
Our { main contributions} are:
\begin{enumerate}[leftmargin=*]
    \item A {\bfseries constant batch-size Stochastic Frank-Wolfe (SFW) algorithm for finite sums with linear prediction}. We describe the method in Section \ref{sec:methods} and discuss its computational and memory cost. 
    \item A {\bfseries non-asymptotic rate analysis  on smooth and convex objectives}. The suboptimality of the SFW algorithm after $t$ iterations can be bounded as $\mathcal{O}\left({\kappa}/{t}\right)$, where $\kappa$ is a data-dependent constant we will discuss later. It is upper bounded by the sample-size $n$ but, depending on the setting, can be potentially much smaller.
    \item An {\bfseries asymptotic analysis for non-convex objectives}.
    We prove that SFW converges to a stationary point for smooth but potentially non-convex functions. This is the first stochastic FW variant that has convergence guarantees in this setting of large practical interest.

\end{enumerate}
Finally, we compare the SFW algorithm with other stochastic Frank-Wolfe algorithms amenable to constant batch size on different machine learning tasks. These experiments show that the proposed method converges at least as fast as previous work, and notably faster on several such instances.

\subsection{Related Work}
We split existing stochastic FW algorithm into two categories: methods with increasing batch size and methods with constant batch size.

\paragraph{Increasing batch size Stochastic Frank-Wolfe.}
This variant allows the number of gradient evaluations to grow with the iteration number \cite{pmlr-v54-goldfarb17a, hazan2016variance, reddi2016stochastic}. Because of the growing number of gradient evaluations, these methods converge towards a deterministic full gradient FW algorithm and so asymptotically share their computational requirements.  In this work we will instead be interested in \emph{constant batch-size} methods, in which the number of gradient evaluations does not increase with the iteration number. See \citet{hazan2016variance} for a detailed comparison of assumptions and complexities for Stochastic Frank-Wolfe methods with increasing batch sizes, in terms of both iterations and gradient calls.

\paragraph{Constant batch size Stochastic Frank-Wolfe.} These methods use a constant batch size $b$, which is chosen by the user as a hyperparameter.
 In the convex and smooth setting, \citet{mokhtari2018stochastic} and \citet{locatello2019} reach $\mathcal O\left({1}/{\sqrt[3]{t}}\right)$ convergence rates.
 The rate of~\citet{locatello2019} further holds for non-smooth and non-Lipschitz objectives.
 \citet{zhang2019one} requires second order knowledge of the objective.
\citet{lu2018generalized} proves convergence for an averaged iterate in $\mathcal{O}(n/t)$ with $n$ the number of samples in the dataset. Let us assume for simplicity that we use unit batch size. Since each iteration involves only one data point, the per-iteration complexity of their method reduces by a factor of $n$ the per-iteration complexity of full-gradient method. On the other hand, the method proposed in this work loses this factor in the rate in number of iterations, reaching the same overall complexity as the deterministic full gradient method. 
Depending on the use-case (large or small datasets), each of the rates reported in \citet{lu2018generalized} and \citet{mokhtari2018stochastic} can have an advantage over the other. In favorable cases, the rate of convergence achieved by our method is nearly independent of the number of samples in the dataset. In these cases, our method is therefore faster than both. In the worst case, it matches the $\mathcal{O}\left({n}/{t}\right)$ bound \cite{lu2018generalized}. 

\subsection{Notation}
Throughout the paper we denote vectors in lowercase boldface letters ($\ww$), matrices in uppercase boldface letters ($\XX$), and sets in calligraphic letters (e.g., $\mathcal{C}$).
We say a function $f$ is $L$-smooth in the norm $\| \cdot\|$ if it is differentiable and its gradient is $L$-Lipschitz continuous with respect to $\| \cdot\|$, that is, if it verifies $\|\nabla f(\xx) - \nabla f(\yy)\|_* \leq L\|\xx - \yy\|$ for all $\xx, \yy$ in the domain (where $\| \cdot\|_*$ is the dual norm of $\| \cdot\|$). For a one dimensional function $f$, this reduces to $|f'(z) - f'(z')| \leq L |z-z'|$ for all $z,~z'$ in the domain.
For the time dependent vector $\uu_t$, we denote by $\uu_t^{(i)}$ its $i$-th coordinate.

We distinguish $\EE$, the full expectation taken with respect to all the randomness in the system, from $\EE_{t}$, the conditional expectation with respect to the random index sampled at iteration $t$,  conditioned on all randomness up to iteration $t$.

Finally, $\LMO(\uu)$ returns an arbitrary element in $\argmin_{\sss\in \CC} \langle \sss, \uu \rangle$.

\section{Methods}
\label{sec:methods}

\subsection{A Primal-Dual View on Frank-Wolfe}

In this subsection, we present the Frank-Wolfe algorithm as an alternating optimization scheme on a saddle-point problem. This point of view motivates the design of the proposed SFW algorithm. This perspective is similar to the two player game point of view of \citet{abernethy2017OnFA, abernethy18a}, which we express using convex conjugacy.
We suppose here that $f$ is closed, convex and differentiable.

Let us rewrite our initial problem \eqref{eq:obj_fun} in the equivalent unconstrained formulation
  \begin{equation}
    \label{eq:fw2}
    \minimize_{\ww \in \RR^d} f(\XX\ww) + \imath_{\mathcal{C}}(\ww)~,
  \end{equation}
  where $\imath_\CC$ is the indicator function of $\CC$: it is $0$ over $\CC$ and $+\infty$ outside of $\CC$. 
  
  We denote by $f^\ast$ the convex conjugate of $f$, that is, $f^\ast(\balpha) \defas \max_{\ww}\langle \balpha, \ww\rangle -f(\ww)$. 
  Whenever $f$ is closed and convex, it is known that $f = (f^\ast)^\ast$, and so we can write $f(\XX\ww) = \max_{\balpha} \{ -f^\ast(\balpha) + \langle \XX\ww, \balpha\rangle\}$. Plugging this identity into the previous equation, we arrive at a saddle-point reformulation of the original problem:

\begin{equation}\label{eq:saddle_reformulation}
\min_{\ww \in \RR^d}\max_{\balpha \in \RR^n} \left\{ \mathcal{L}(\ww, \balpha) \defas - f^\ast(\balpha) +  \imath_{\mathcal{C}}(\ww) + \langle \XX\ww, \balpha\rangle \right\}.
\end{equation}

This reformulation allows to derive the Frank-Wolfe algorithm as an alternating optimization method on this saddle-point reformulation. To distinguish the algorithm in this section from the stochastic algorithm we propose, we denote the iterates in this section by $\bar\balpha_t$, $\bar\ww_t$.

The first step of the Frank-Wolfe algorithm is to compute the gradient of the objective at the current iterate. In the saddle-point formulation, this corresponds to maximizing over the dual variable $\balpha$ at step $t$:

\begin{align}
&\bar\balpha_{t} \in
 \argmax_{\balpha \in \RR^n} \left\{\mathcal{L}(\bar\ww_{t-1}, \balpha) =  - f^\ast(\balpha) + \langle \XX\bar\ww_{t-1}, \balpha\rangle\right\} \notag\\
 &\iff \bar\balpha_t = \nabla f(\XX\bar\ww_{t-1}).\label{eq:exact_dual_minimization}
 \end{align}
 
Then, the LMO step corresponds to fixing the dual variable and minimizing over the primal one $\ww$. This gives 
 
 \begin{align}
&\bar\sss_{t} \in \argmin_{\ww \in \RR^d} \left\{\mathcal{L}(\ww, \bar\balpha_t) = \imath_{\CC}(\ww) + \langle \ww,\XX\tran \bar{\balpha_t} \rangle \right\} \notag\\
&\iff \bar\sss_t = \LMO(\XX\tran\bar\balpha_t).
\end{align}

Note that from the definition of the $\LMO$,  $\bar\sss_t$ can always be chosen as an extreme point of the constraint set $\CC$. We then update our iterate using the convex combination
\begin{align}
\bar\ww_{t} = (1 - \gamma_t)\bar\ww_{t-1} + \gamma_t \bar \sss_t,
\end{align}

where $\gamma_t$ is a step-size to be chosen.
These updates determine the Frank-Wolfe algorithm.

\subsection{The Stochastic Frank-Wolfe Algorithm}

\begin{algorithm}[tb]
   \caption{Stochastic Frank-Wolfe} \label{alg:sfw}
\begin{algorithmic}[1]
   \STATE {\bfseries Initialization:} $\ww_0\in\CC$, $\balpha_0 \in \RR^n$, ${\rr_0 =\XX^\top\balpha_0}$

    \FOR{$t=1, 2, \dots, $}
        \STATE Sample $i \in \{1, \ldots, n\}$ uniformly at random. \label{lst:line:sample}
        \STATE Update $\balpha_t^{(i)} = \frac{1}{n}f'_i(\xx_i\tran\ww_{t-1})$
        \STATE Update $\balpha_t^{(j)} = \balpha_{t-1}^{j}$, $j\neq i$
        \label{lst:line:refresh_alpha}
       
        \STATE $\rr_{t} = \rr_{t-1} + (\balpha_{t}^{(i)} - \balpha_{t-1}^{(i)}) \xx_i$ \label{lst:line:refresh_gradient}
    
        \STATE $\sss_t = \LMO(\rr_t)$  \label{lst:line:lmo}
        
        \STATE $\ww_{t} = \ww_{t-1} + \frac{2}{t+2}(\sss_t - \ww_{t-1})$  \label{lst:line:update_iterate}
   \ENDFOR
\end{algorithmic}
\end{algorithm}
We now consider a variant in which we replace the exact minimization of the dual variable \eqref{eq:exact_dual_minimization} by a minimization over a single coordinate, chosen uniformly at random.

Let us define the function $f$ from $\RR^n$ to $\RR$ as $f(\btheta) \defas \frac{1}{n} \sum_{i=1}^n f_i(\btheta_i)$. We can write our original optimization problem as an optimization over $\ww\in\CC$ of $f(\XX\ww)$. Still alternating between the primal and the dual problems, we replace maximization over the full vector $\balpha$ in \eqref{eq:exact_dual_minimization} with optimization along the coordinate $i$ only. We obtain the update $\balpha_{t}^{(i)} = \frac{1}{n}f_i'(\xx_i\tran\ww_{t-1})$.
Doing so changes the cost per-iteration from $\mathcal{O}(nd)$ to  $\mathcal{O}(d)$, and yields Algorithm~\ref{alg:sfw}.

We now describe our main contribution, Algorithm \ref{alg:sfw} (SFW) above. It follows the classical Frank-Wolfe algorithm, but replaces the gradient with a stochastic estimate of the gradient.

Throughout Algorithm \ref{alg:sfw}, we maintain the following iterates:
\begin{itemize}
    \item the iterate $\ww_t$,
    \item the stochastic estimator of $\nabla f(\XX\ww_{t-1})$ denoted by $\balpha_t \in \RR^n$,
    \item the stochastic estimator of the full gradient of our loss $\XX\tran\nabla f(\XX\ww_{t-1})$,  denoted by $\rr_t \in \RR^d$.
\end{itemize}

\textbf{Algorithm.} At the beginning of iteration $t$, we have access to $\balpha_{t-1}$, $ \rr_{t-1}$ and to the iterate $\ww_{t-1}$.

Thus equipped, we sample an index $i$ uniformly at random over $\{1,\dots, n\}$. We then compute the gradient of our loss function for that datapoint, on our iterate, yielding $[\nabla f(\XX\ww_{t-1})]_i = \frac{1}{n}f'_i(\xx_i\tran \ww_{t-1})$. We update the stochastic gradient estimator $\balpha_t$ by refreshing the contribution of the $i$-th datapoint and leaving the other coordinates untouched.

\begin{remark}
Coordinate $j$ of our estimator $\balpha_t$ contains the latest sampled one-dimensional derivative of $\frac{1}{n}f_j$.
\end{remark}

To get $\rr_t$, we do the same, removing the previous contribution of the $i$-th datapoint, and adding the refreshed contribution. This allows us not to store the full data-matrix in memory.

The rest of the algorithm continues as the deterministic Frank-Wolfe algorithm from the previous subsection: we find the update direction from $\sss_t = \LMO(\rr_t)$, and we update our iterate using a convex combination of the previous iterate $\ww_{t-1}$ and $\sss_t$, whereby our new iterate is feasible.

\begin{remark}
Our algorithm requires to keep track of the $\balpha_t$ vector and amounts to keeping one scalar per sample in memory. Our method requires the same small memory caveat as other variance reduced algorithms such as SDCA \cite{shalev2013stochastic}, SAG \cite{Schmidt2013MinimizingFS} or SAGA \cite{defazio14SAGA}. 
Despite the resemblance of our gradient estimator to the Stochastic Average Gradient \cite{Schmidt2013MinimizingFS}, the convergence rate analyses are quite different.
\end{remark}

\section{Analysis}
\label{sec:analysis}

\subsection{Preliminary tools}
\input{preliminary_tools.tex}

\subsection{Worst-Case Convergence Rates for Smooth and Convex Objectives}

We state our main result in the $L$-smooth, convex setting. 
In this section, we suppose that  the $f_i$s are $L$-smooth and convex and that for all $\btheta$, $f(\btheta)=\frac{1}{n}\sum_{i=1}^n f_i(\btheta_i)$. The objective function $f$ then satisfies \eqref{norman} as noted previously.

\begin{theorem}\label{theorem:convex_rate}
Let $H_0 \defas \|\balpha_{0} - \nabla f(\XX\ww_{0})\|_1$ be the initial error of our gradient estimator and $\ww_\star\in\CC$ a solution to \ref{eq:obj_fun}. We run Algorithm \ref{alg:sfw} with step sizes $\gamma_t=2/(t+2)$.
At time-step $t\geq 2$, the expected primal suboptimality $\EE\varepsilon_t = {\EE[f(\XX\ww_t) - f(\XX\ww_\star)]}$ has the following upper bound

\begin{align}
\label{eq:convex_rate}
\bsp
    \EE\varepsilon_t \leq& 2L\left( \frac{D^2_2 + 4(n-1)D_1D_\infty}{n}  \right) \frac{t}{(t+1)(t+2)} \\
    &+  \frac{2 \varepsilon_0 + (2 D_\infty H_0 + 64 {L D_1 D_\infty}) n^2}{(t+1)(t+2)}
\esp
\end{align}

\end{theorem}

\begin{remark}
The rate of the proposed method in terms of gradient calls is also given by \eqref{eq:convex_rate} (one gradient call per iteration), whereas for deterministic Frank-Wolfe, the (deterministic) suboptimality after $t$ gradient calls has the following upper bound~\citep{jaggi2013revisiting, hazan2016variance}

\begin{equation}
     \varepsilon_t \leq  \frac{2 L D_2^2}{t}\,.
\end{equation}
\end{remark}

In this paper, we will only discuss unit batch size. We can adapt our algorithm and proofs to consider sampling a mini-batch of $b$ datapoints at each step. The leading term in our rate from Theorem~\ref{theorem:convex_rate} will change: we will use $\rho=1-\mfrac{b}{n}$ in Lemma~\ref{lemma:ht_asymptotics}. The overall rate will be modified accordingly. The per-iteration complexity will then become $\mathcal{O}(bd)$.

We first \textbf{sketch the outline of the proof} before delving into details.
The proof of this convergence rate builds on three key lemmas.
The first is an adaptation of Lemma~2 of \citet{mokhtari2018stochastic} which bounds the suboptimality at step $t$ by the sum of a contraction in the suboptimality at $t-1$, a vanishing term due to smoothness, and a last term depending on our gradient estimator's error in $\ell_1$ norm. The first two terms show up in the convergence proof of the full-gradient Frank-Wolfe, see \citet{lacoste2015global}. The last term is an error, or noise term. Supposing the error term vanishes fast enough, we can fall back on the full-gradient proof technique  \cite{frank1956algorithm, jaggi2013revisiting}.

From there, we show that the error term verifies a particular recursive inequality in lemma \ref{lemma:ht_upper_bound}. In lemma \ref{lemma:ht_asymptotics}, we then leverage this inequality to prove that the error term vanishes as $\mathcal O(1/t)$, finally allowing us to obtain the promised rate.
The formal statements of these lemmas follow.

\begin{lemma}
\label{lemma:sufficient_decrease}
Let $f_i$ be convex and $L$-smooth for all $i$. For \emph{any} direction $\balpha_t\in \RR^n$, define $\sss_t = \LMO(\XX\tran \balpha_t)$, $\xx_t = (1-\gamma_t)\xx_{t-1} + \gamma_t \sss_t$ and $H_t = \|\balpha_t - \nabla f(\XX\ww_{t-1})\|_1$.

We have the following upper bound on the primal suboptimality at step $t$:
\begin{align}
\bsp
    \varepsilon_t  \leq~&  (1-\gamma_t)\varepsilon_{t-1}  + \gamma_t^2\frac{LD^2_2}{2n} + {\color{brown}\underbrace{\gamma_t D_\infty H_t}_{\text{error term}}}\,. \label{eq:subopt_last}
\esp
\end{align}

\end{lemma}

We defer this proof to Appendix~\ref{apx:sufficient_decrease}.

\begin{remark}
This lemma holds for \emph{any} direction $\balpha_t\in\RR^n$, not necessarily the $\balpha_t$ given by the SFW algorithm. 
\end{remark}

\begin{remark}
This lemma generalizes the key inequality used in many proofs in the Frank-Wolfe literature \cite{jaggi2013revisiting} but includes an extra {\color{brown} error term} to account for the fact that the direction  $\balpha_t$, which we use for the $\LMO$ step and therefore to compute the updated iterate, is \emph{not} the true gradient.
If $\balpha_t = \nabla f(\XX\ww_{t-1})$, that is, if we compute the gradient on the full dataset, then $H_t=0$ and we recover the standard quadratic upper bound.
\end{remark}

In the following, $\balpha_t$ is the direction given by Algorithm~\ref{alg:sfw}, and the $\ell_1$ error term is in terms of that $\balpha_t$: 

\begin{align}
    H_t \defas \|\balpha_t - \nabla f(\XX\ww_{t-1})\|_1
\end{align}
for $t > 0$ and $H_0 = \|\balpha_0 - \nabla f(\XX\ww_{0})\|_1$.

Notice that we define the gradient estimator's error with the $\ell_1$ norm. The previous lemma also holds with the $\ell_2$ norm of the gradient error (replacing $D_\infty$ by $D_2$). We prefer the $\ell_1$ norm because of the finite-sum assumption: it induces a coordinate-wise separation over $\balpha_t$ which corresponds to a datapoint-wise separation. The following lemma crucially leverages this assumption to upper bound $H_t$ given by the SFW algorithm.

\begin{lemma}\label{lemma:ht_upper_bound}
For the stochastic gradient estimator $\balpha_t$ given by Algorithm \ref{alg:sfw} (SFW), we can upper bound $H_t = {\|\balpha_t - \nabla f(\XX\ww_{t-1})\|_1}$ in conditional expectation as follows
\begin{align}
    \EE_t H_t \leq \left(1-\mfrac{1}{n}\right)\left(H_{t-1} + \gamma_{t-1}\frac{LD_1}{n}\right). \label{eq:ht_upper_bound}
\end{align}
 
\end{lemma}

\begin{proof}
We have the following expression for $\balpha_t$, supposing that index $i$ was sampled at step $t$.

\begin{align}
    \balpha_t &= \balpha_{t-1} + \left(\frac{1}{n}f'_i(\xx_i\tran\ww_{t-1}) - \balpha_{t-1}^{(i)}\right)\ee_i
\end{align}
where $\ee_i$ is the $i$-th vector of the canonical basis of $\RR^n$.
Consider a fixed coordinate $j$. Since there is a $\frac{1}{n}$ chance of $\balpha_j$ being updated to $f'_j(\xx_j^\top \ww_{t-1})$, taking conditional expectations we have

\begin{align}
\EE_t H_t^{j} &\defas |\balpha_t^{(j)} - \frac{1}{n} f_j'(\xx_j^\top \ww_{t-1})| \\
&= \left(1 - \frac{1}{n} \right)|\balpha_{t-1}^{(j)} -  \frac{1}{n}f_j'(\xx_j^\top \ww_{t-1})|.
\end{align}

Summing over all coordinates we then have
\begin{align}
    \EE_t H_t &=  \sum_{j=1}^n \EE_t H_t^j \\
    &= \left(1-\frac{1}{n}\right) \underbrace{ \left\|\balpha_{t-1} - \nabla f(\XX \ww_{t-1})\right\|_1}_{\delta_{t-1}}.
\end{align}

\looseness=-1We denote the $\ell_1$ norm term by $\delta_{t-1}$ for ease.
Let us introduce the full gradient at the previous step $\nabla f(\XX\ww_{t-2})$ and use the triangle inequality. Our finite sum assumption gives us that for all $j\in \{1,\dots, n\}$ and $\ww\in \CC$, $[\nabla f(\XX\ww)]_j = \frac{1}{n} f_j'(\xx_j\tran\ww)$. Then, we separate the $\ell_1$ norm, use \mbox{$L$-smoothness} of each of the $f_j$s and the definition of $\ww_{t-1}$. 
\begin{align}
    \delta_{t-1} &\leq H_{t-1} + \|\nabla f(\XX\ww_{t-2}) -\nabla f(\XX\ww_{t-1}) \|_1
    \\
    &\leq H_{t-1} + \frac{L}{n}\sum_{j=1}^n|\xx_j\tran(\ww_{t-1} - \ww_{t-2})| \\
    &\leq H_{t-1} + \gamma_{t-1} \frac{L}{n}\sum_{j=1}^n |\xx_j\tran (\sss_{t-1} - \ww_{t-2})| \\
    &\leq H_{t-1} + \gamma_{t-1} \frac{L}{n} \|\XX(\sss_{t-1} - \ww_{t-2})\|_1
\end{align}
where we used $\ww_{t-1} - \ww_{t-2} = \gamma_{t-1}(\ww_{t-1} - \sss_{t-2})$. Finally, using the definition of the diameter $D_1$, we obtain inequality \eqref{eq:ht_upper_bound}.
\end{proof}
Now, we can use the structure of this recurrence to obtain the desired rate of convergence for our gradient estimator. We state this in the following lemma.

\input{lemma_3_fix}

The remainder of the proof of Theorem~\ref{theorem:convex_rate} follows the usual Frank-Wolfe proofs in the full gradient case, which can be found e.g. in \citet{frank1956algorithm, jaggi2013revisiting}.
 Here is a brief sketch of these steps: we tie the three key lemmas together, plugging in the bound on $\EE H_t$ given by Lemma \ref{lemma:ht_asymptotics} into the upper bound on the suboptimality at step $t$ given by Lemma \ref{lemma:sufficient_decrease}. By specifying the step size $2/(t+2)$, and scaling the bounds by a factor of $(t+1)(t+2)$, we obtain a telescopic sum, allowing us to upper bound the expected suboptimality at the latest step considered. The details are deferred to Appendix \ref{apx:tying_up}.

\subsection{Worst-case Convergence Rates for Smooth, Non-Convex Objectives}
\label{sec:convergence_gap}
We start by recalling the definition of the Frank-Wolfe gap:
\begin{align}
    g_t &= \max_{\sss \in \CC} \langle \nabla f(\XX\ww_{t-1}), \XX(\ww_{t -1} - \sss)\rangle.
\end{align}

Previous work \cite{jaggi2013revisiting} has shown the importance of the Frank-Wolfe gap. In the convex setting, it is a primal-dual gap, and as such, upper bounds both primal and dual suboptimalities. In the general non-convex setting, it is a measure of near-stationarity. We define a stationary point as any point $\ww_\star$  such that for all $\ww\in \CC$,  ${\langle \nabla f(\XX\ww_\star), \XX(\ww - \ww_\star)\rangle} \geq 0$~\citep{bertsekas1999nonlinear}. From this definition, it is clear that the Frank-Wolfe gap $g_t$ is zero only at a stationary point.

In this section, we suppose that $f_i$ is $L$-smooth for $i$ in $\{1,\dots, n\}$, but not necessarily convex. The following theorem states that we can still obtain a stationary point from Algorithm \ref{alg:sfw}.

\begin{theorem}
\label{thm:nonconvex_rate}
Let $\ww_t$ be computed according to Algorithm~\ref{alg:sfw}, then 
\begin{align}
    \liminf_{t\to\infty} \,\EE_t g_t = 0,
\end{align}

where $g_t$ is the Frank-Wolfe gap.

\end{theorem}

The proof of this result is deferred to Appendix \ref{apx:proof_non_convex}.

\section{Stopping Criterion}
In this section, we define a natural stochastic Frank-Wolfe gap, and explain why it can be used as a stopping criterion.

We recall the definition of the true Frank-Wolfe gap $g_t$, and define the stochastic Frank-Wolfe gap $\hat g_t$ as:
\begin{align}
    g_t &= \max_{\sss \in \CC} \langle \nabla f(\XX\ww_{t-1}), \XX(\ww_{t -1} - \sss)\rangle ,\\
    \hat g_t&= \max_{\sss \in \CC} \langle \balpha_t, \XX(\ww_{t -1} - \sss)\rangle 
\end{align}

for $\balpha_t$ given by SFW.

The Frank-Wolfe gap's properties make estimating it very desirable: when the gap is small for a given iteration of a Frank-Wolfe type algorithm, we can guarantee we are close to optimum (or to a stationary point in the general non-convex case).
Unfortunately, in datasets with many samples, and since it depends on the full gradient, computing this gap can be impractical. 

The following proposition shows that the stochastic Frank-Wolfe gap estimator resulting from Algorithm \ref{alg:sfw} can be used as a proxy for the true Frank-Wolfe gap.

\begin{proposition}
\label{prop:proxy_gap_informal}
For $\balpha_t$ given by Algorithm \ref{alg:sfw}, we can bound the distance between the stochastic Frank-Wolfe gap and the true Frank-Wolfe gap as follows: 

\begin{align}
    |g_t - \hat{g}_t | \leq D_\infty H_t,
\end{align}

which yields the following bound in expectation

\begin{align}
    \EE |g_t - \hat g_t| \leq~& 2\mfrac{L D_1 D_\infty}{n}\left(\mfrac{2(n-1)}{t+2} + \left(1-\mfrac{1}{n}\right)^{t/2}\log t\right) \nonumber\\
    &+ \left(1-\mfrac{1}{n}\right)^t D_\infty H_0.
\end{align}
\end{proposition}

We defer the proof to Appendix \ref{apx:proof_proxy_gap}.

If $\hat{g}_t$ goes to 0, then the true Frank-Wolfe gap will be expected to vanish as well. We therefore propose to use $\hat{g}_t$, which is computed as a byproduct of our SFW algorithm, as a heuristic stopping criterion, but defer a more in-depth theoretical and empirical analysis of this gap to future work.

\section{Discussion}
\label{sec:discussion}
\input{experiments_table.tex}
\input{experiments_batch_size1_plots.tex}

In this section, we compare the convergence rate of the proposed SFW, \citet{lu2018generalized} and \citet{mokhtari2018stochastic} as shown in Table~\ref{tab:rates}. We use big $\mathcal O$ notation, only focusing on dependencies in $n$ and $t$ to upper bound the suboptimality at step $t$.

To make a fair comparison, including dependencies in $n$, the number of samples, we first standardize notations across papers. \citet{lu2018generalized} use the same formal setting as ours, where $\xx_i\tran\ww$ is the argument to the $i$-th objective $f_i$, and the full objective is the average of these. \citet{mokhtari2018stochastic} set themselves in a more general setting, where they only assume access to an \emph{unbiased estimator} of the full gradient. 

For ease of comparison, we rewrite the two algorithms of \citet{lu2018generalized} and \citet{mokhtari2018stochastic} in Appendix~\ref{apx:comparison_other_methods} using our notations.

Because of their more general setting, the $L_{\text{mok}}$ Lipschitz constant appearing in \citet{mokhtari2018stochastic} can be written $L_{\text{mok}} = \frac{L}{n}n\max_i\|\xx_i\|_2$ (using Cauchy-Schwartz). Their diameter constant $D_\text{mok} = \max_{\uu, \vv\in \CC}\|\uu-\vv\|_2$  is also independent of $n$. Finally, their $\sigma^2$ term controlling the variance of their stochastic estimator should also be $n$-independent.
Under this notation, their convergence rate (Theorem 3, \citet{mokhtari2018stochastic}) is $\mathcal{O}\left(1/\sqrt[3]{t}\right)$ with no dependency in $n$ as expected.

\citet{lu2018generalized} have a detailed discussion of the rate of their method, and achieve the overall rate of $\mathcal{O}\left({n}/{t}\right)$. 

To fairly compare these rates to the one given by Theorem~\ref{theorem:convex_rate}, we must consider the $D_1$ and $D_\infty$ terms, which may depend on the number of samples $n$. The rate we obtain has a leading term of $\mathcal{O}\left({D_1 D_\infty}/{t}\right)$, and a second term of $\mathcal{O}\left({D_1D_\infty n^2}/{t^2}\right)$. 
The second term is dominated by the first in the regime $t > n^2$.
Defining $\kappa = D_1/D_\infty$, we can write $D_1 D_\infty$ as $\kappa D_\infty^2$. We have that $\kappa \leq n$, meaning that in the worst case, this bound matches the one in \citet{lu2018generalized}. When the constraint set is the $\ell_1$ ball $\{\ww\,|\, \|\ww\|_1 \leq \lambda\}$, we have the following closed form expression:

\begin{align}
    \kappa = \frac{\|X\|_{1,1}}{\|X\|_{1, \infty}} = \frac{\max_j \sum_{i=1}^n |X_{ij}|}{\max_{ij}|X_{ij}|}.
\end{align}
We can therefore easily compute it for given datasets. 
\begin{remark}
We briefly remark that if for every feature, the contribution of that feature is limited to a few datapoints, this ratio will be small, and therefore the overall bound does not depend on the number of samples. This tends to happen for TF-IDF text representations, and for fat-tailed data.
\end{remark}
Formal analysis of this ratio exceeds the scope of this paper, and we defer it to future work. We report values of $\kappa$ for the considered datasets in Section~\ref{sec:experiments}.

\section{Implementation Details}
Our implementation is available in the C-OPT package.\footnote{\url{https://github.com/openopt/copt}}

\textbf{Initialization.} We use the cheapest possible initialization: our initial stochastic gradient estimator $\balpha_0$ starts out at $0$. We also then have that $\rr_0=0$.

\textbf{Sparsity in $\XX$.} Suppose there are at most $s$ non-zero features for any datapoint $\xx_i$. Then for instances where $\cal C$ is an $\ell_1$ ball, all updates in SFW algorithm can be implemented using using only the support of the current datapoint, making the per-iteration cost of SFW $\mathcal{O}(s)$ instead of $\mathcal{O}(d)$. Large-scale datasets are often extremely sparse, so leveraging this sparsity is crucial. For example, in the LibSVM datasets suite, 8 out of the 11 datasets with more than a million samples have a density between $10^{-4}$ and $10^{-6}$. 

\section{Experiments}
\label{sec:experiments}

We compare the proposed SFW algorithm with other constant batch size algorithms from \citet{mokhtari2018stochastic} and \citet{lu2018generalized}.

\paragraph{Experimental Setting.} We consider $\ell_1$ constrained logistic regression problems on the \textsc{Breast Cancer} and \textsc{RCV1} datasets, and an $\ell_1$ constrained least squares regression problem on the \textsc{California Housing} dataset, all from the UCI dataset repository \cite{ucidatasets}. See Table \ref{tab:datasets} for details and links. 

We compare the relative suboptimality computed for each method, given by $(f(\XX\ww_t) - f_{\min})/(f_{\max} - f_{\min})$ at step $t$, where $f_{\min}$ and $f_{\max}$ are the smallest and largest function values encountered by any of the compared methods. We compute these values at different time intervals (the same for each method) depending on problem size, to limit the time of each run.
We use batches using $1\%$ of the dataset at each step, following \citet{lu2018generalized}. Within a batch, data points are sampled without replacement.

We plot these values as a function of the number of gradient evaluations, equal to the number of iterations times the batch size $b$: for all of the considered methods, an iteration involves exactly $b$ gradient evaluations and one call to the LMO. This allows us to fairly compare the convergence speeds in practice.

Compared to both methods from \citet{mokhtari2018stochastic} and \citet{lu2018generalized}, the proposed SFW achieves lower suboptimality for a given number of iterations on the considered tasks and datasets. We have no explanation for the initial regime in the \textsc{California Housing} dataset, before the methods start showing what resembles a sublinear rate, as the theory prescribes. Notice that the \textsc{RCV1} dataset has the lowest $\kappa/n$ (due to sparsity of the TF-IDF represented data), and that the method presented in this paper performs particularly well on this dataset.

{\bfseries Comparison with \citet{mokhtari2018stochastic}.} Although the step-size in our SFW Algorithm and the one proposed in the paper are of the same order of magnitude $\mathcal{O}(1/t)$, \citet{mokhtari2018stochastic} use $f'_i(\xx_i\tran\ww_{t-1})$ instead of our $(1/n)f'_i(\xx_i\tran\ww_{t-1})$, because they require an unbiased estimator. Their choice induces higher variance, which then requires the algorithm to use momentum with a vanishing step size in their stochastic gradient estimator, damping the contributions of the later gradients (using $\rho_t=\frac{1}{t^{2/3}}$, see the pseudo code in Appendix~\ref{apx:comparison_other_methods}). This may explain why the method proposed in \citet{mokhtari2018stochastic} achieves slower convergence. On the contrary, the lower variance in our estimator $\balpha_t$ allows us to give the same weight to contributions of later gradients as to previous ones, and to forget all but the last gradient computed at a given datapoint.

{\bfseries Comparison with \citet{lu2018generalized}.} The method from \citet{lu2018generalized} computes the gradient at an averaged iterate, putting more weight on earlier iterates, making it more conservative. This may explain slower convergence versus the SFW algorithm proposed in this paper in certain settings.

\section{Conclusion and Future Work}

Similarly to methods from the Variance Reduction literature such as SAG, SAGA, SDCA, we propose a Stochastic Frank Wolfe algorithm tailored to the finite-sum setting. Our method achieves a step towards attaining comparable complexity iteration-wise to deterministic, true-gradient Frank-Wolfe in the smooth, convex setting, at a per-iteration cost which can be nearly independent of the number of samples in the dataset in favorable settings. Our rate of convergence depends on the norm ratio $\kappa$ on the dataset, which is related to a measure of the weights of the data distribution's tails. We will explore this intriguing fact in future work.

We propose a stochastic Frank-Wolfe gap estimator, which may be used as a heuristic stopping criterion, including in the non-convex setting. Its distance to the true gap may be difficult to evaluate numerically. Obtaining a practical bound on this distance is an interesting avenue for future work.

\citet{guelat1986some} and \citet{lacoste2015global} have proposed variants of the FW algorithm that converge linearly on polytope constraint sets for strongly convex objectives: the Away Steps Frank-Wolfe and the Pairwise Frank-Wolfe. \citet{pmlr-v54-goldfarb17a} studied stochastic versions of these and showed linear convergence over polytopes using increasing batch sizes. Our SFW algorithm, the natural stochastic gap and the analyses in this paper should be amenable to such variants as well, which we plan to explore in future work. 

\section*{Acknowledgments}

The authors would like to thank Donald Goldfarb for early encouragement in this direction of research, and Armin Askari, Sara Fridovich-Keil, Yana Hasson, Thomas Kerdreux, Nicolas Le Roux, Romain Lopez, Gr\'egoire Mialon, Courtney Paquette, Hector Roux de B\'ezieux, Alice Schoenauer-Sebag, Dhruv Sharma, Yi Sun, and Nilesh Tripuraneni for their constructive criticism on drafts of this paper. The authors also warmly thank Maria-Luiza Vladareanu for finding and reporting an error in an earlier draft's proof, and Alex Belloni, Jose Moran for discussions as well.

Francesco Locatello is supported by the Max Planck ETH Center for Learning Systems, by an ETH core grant (to Gunnar R\"atsch), and by a Google Ph.D. Fellowship.  Robert Freund’s research is supported by AFOSR Grant No. FA9550-19-1-0240.

\bibliography{biblio}
\bibliographystyle{icml2020}

\input{appendix.tex}

\end{document}

%% file: defs.tex
\usepackage{dsfont}
\usepackage{enumitem}
\usepackage{titlesec}
\usepackage{amsmath} 
\usepackage{amssymb}
\usepackage{amsfonts}    
\usepackage{amsthm}
\usepackage{empheq}
\usepackage{xcolor, color, colortbl}
\usepackage{listings}
\usepackage[tikz]{bclogo}

\newcommand{\floor}[1]{\left\lfloor #1 \right\rfloor}

\definecolor{codegreen}{rgb}{0,0.6,0}
\definecolor{codegray}{rgb}{0.5,0.5,0.5}
\definecolor{codepurple}{rgb}{0.58,0,0.82}
\definecolor{backcolour}{rgb}{0.95,0.95,0.92}
 
\lstdefinestyle{mystyle}{
    backgroundcolor=\color{backcolour},   
    commentstyle=\color{codegreen},
    keywordstyle=\color{magenta},
    numberstyle=\tiny\color{codegray},
    stringstyle=\color{codepurple},
    basicstyle=\ttfamily\footnotesize,
    breakatwhitespace=false,         
    breaklines=true,                 
    captionpos=b,
    keepspaces=true,
    numbers=left,
    numbersep=5pt,
    showspaces=false,
    showstringspaces=false,
    showtabs=false,                  
    tabsize=2
}
 
\definecolor{shadecolor}{RGB}{242, 242, 242}

\definecolor{mydarkblue}{rgb}{0,0.08,0.45}

\graphicspath{{./figures/}}

\usepackage{nccmath} 
\definecolor{mydarkblue}{rgb}{0,0.08,0.45}
\definecolor{myblue}{HTML}{D2E4FC}
\newcommand*\mybluebox[1]{%
\colorbox{myblue}{\hspace{1em}#1\hspace{1em}}}

\definecolor{mygreen}{HTML}{008000}

\definecolor{myorange}{HTML}{FFC48C}

\definecolor{Gray}{gray}{0.92}

\newcommand{\red}{\color{red}}

\DeclareMathOperator*{\argmin}{arg\,min}
\DeclareMathOperator*{\argmax}{arg\,max}

\DeclareMathOperator*{\minimize}{minimize}

\DeclareMathOperator{\LMO}{LMO}

\usepackage{pifont}
\newcommand{\xmark}{\ding{55}}%

\def\bsp{\begin{split}}
\def\esp{\end{split}}

\def\balpha{\boldsymbol u}
\def\balpha{\boldsymbol \alpha}

\def\btheta{\boldsymbol \theta}
\def\bsigma{\boldsymbol \sigma}

\def\ee{\boldsymbol e}
\def\rr{\boldsymbol r}
\def\xx{\boldsymbol x}
\def\ww{\boldsymbol w}

\def\uu{\boldsymbol u}
\def\vv{\boldsymbol v}
\def\yy{\boldsymbol y}
\def\sss{\boldsymbol s}

\def\CC{\mathcal C}

\def\XX{\boldsymbol X}

\def\RR{\mathbb R}
\def\EE{\mathbb E}

\def\defas{\stackrel{\text{def}}{=}}
\def\tran{^\top}

\def\balpha{{\boldsymbol \alpha}}
\def\defas{\stackrel{\text{def}}{=}}

\newtheorem{theorem}{Theorem}

\newtheorem{lemma}{Lemma}
 \newtheorem{proposition}{Proposition}
 \newtheorem{remark}{Remark}


\newcounter{parentnumber}

%% file: preliminary_tools.tex
Recall that in our setting, our objective function is $\ww\mapsto f(\XX \ww)$, where $f(\btheta) = \frac{1}{n}\sum_{i=1}^n f_i(\btheta_i)$. We suppose that for all $i$, $f_i$ is $L$-smooth, which then implies that $f$ satisfies the following  non-standard smoothness condition:
\begin{align}\label{norman}
\|\nabla f(\btheta) - \nabla f(\bar\btheta)\|_p\le \frac{L}{n}\|\btheta - \bar\btheta\|_p \  
\end{align}

for every $p \in [1,\infty]$. Note that in this inequality -- unlike in the standard definition of $L$-smoothness with respect to the $\ell_p$ norm -- the same norm appears on both sides of the inequality. This inequality is proven in Appendix~\ref{apx:proof_smoothness}.  In particular it follows from \eqref{norman} that $f$ is $(L/n)$-smooth with respect to the $\ell_2$ norm.

We therefore have the following quadratic upper bound on our objective function $f$, valid for all $\ww, \vv$ in the domain:
\begin{align}
\bsp
    f(\XX \ww) &\leq f(\XX\vv) + \langle \nabla f(\XX\vv), \ww-\vv \rangle\\
    &\qquad + \frac{L}{2n}\|\XX\left(\ww-\vv\right)\|_2^2~.
    \label{eq:key_ineq}
    \esp
\end{align}

For $p\in \{1,2, \infty\}$, we define the diameters

\begin{align}
    D_p &= \max_{\uu, \vv\in \CC}\|\XX(\uu-\vv)\|_p \label{eq:diam_p}.
\end{align}

\begin{remark}
For $p\in \{1, 2\}$, we have that $D_p^p \leq n D_\infty^p $. 
\end{remark}

%% file: lemma_3_fix.tex
\begin{lemma}\label{lemma:ht_asymptotics}
Let $\gamma_t = \frac{2}{t+2}$. We have the following bound on the expected error $\EE H_t$, for $t\geq 2$: 

\begin{align}\label{rainy}
    \EE H_t \leq~& 2\mfrac{LD_1}{n}\left(\mfrac{2(n-1)}{t+2} + \left(1-\mfrac{1}{n}\right)^{t/2}\log t\right) \nonumber\\
    &+ \left(1-\mfrac{1}{n}\right)^t H_0.
\end{align}

\end{lemma}

\begin{remark}
Our gradient estimator's error in $\ell_1$ norm goes to zero as $O\left(\frac{D_1}{t}\right)$. This rate depends on the assumption of the separability of $f$ into a finite sum of $L$-smooth $f_i$'s. On the other hand, it does not require that each (or any) $f_i$ be convex.
\end{remark}

\begin{proof}
Consider a general sequence of nonnegative numbers, $u_0, u_1, u_2,\ldots,u_t\in\RR_+$ where for all $t$, the following recurrence holds:

\begin{align}
    u_t \leq \rho\left(u_{t-1} + \frac{K}{t+1}\right)
\end{align}
where $0<\rho<1$ and $K>0$ are scalars.

First note that all the iterates are nonnegative.
Suppose $t\geq 2$,\begin{align*}
    u_t &\leq \rho^t u_0 + K\sum_{k=1}^t \frac{\rho^{t-k+1}}{k+1} \\
    & = \rho^t u_0 + K \left(\sum_{k=1}^{\floor{t/2}}\frac{\rho^{t-k+1}}{k+1} + \sum_{k=\lfloor{t/2}\rfloor +1}^t\frac{\rho^{t-k+1}}{k+1}\right) \\
    &\leq \rho^t u_0 + K \left( \sum_{k=1}^{\floor{t/2}} \frac{\rho^{{t/2}}}{k+1} +  \sum_{k=\floor{t/2}+1 }^t 2~\frac{\rho^{t-k+1}}{t+2} \right).
\end{align*}

To go from the second line to the third line, we observe that for ``old" terms with large steps sizes, we are saved by the higher power in the geometric term.
For the more recent terms, the step-size is small enough to ensure convergence. More formally, in the early terms ($1\leq k\leq \floor{t/2}$), we upper bound $\rho^{t-k+1}$ by $\rho^{{t/2}}$. In the later terms ($\floor{t/2}+1\leq k\leq t$), we upper bound $\frac{1}{k+1}$ by $\frac{2}{t + 2}$.

To obtain the full rate, we now study both parts separately. For the first part, we use knowledge of the harmonic series:

\begin{align}
    \rho^{{t/2}} \sum_{k=1}^{\floor{t/2}} \frac{1}{k+1} \leq \rho^{t/2}\log\left(\frac{t}{2} + 1\right)
\end{align}
for $t\geq 2$, we can upper bound $\log\left(\frac{t}{2}+1\right)$ by $\log t$.

For the second part, we use knowledge of the geometric series:

\begin{align}
    \sum_{k=\floor{t/2} +1}^t \rho^{t-k+1} &\leq \frac{\rho}{1-\rho}.
\end{align}

Finally, for $t\geq 2$
\begin{align}
    0\leq u_t \leq K \left(\frac{\rho}{(1-\rho)}~\frac{2}{(t+2)} + \rho^{t/2}\log t \right) +  \rho^t u_0.
\end{align}

The expected error $\EE H_t$ verifies our general conditions with $u_0 = H_0 = \|\balpha_0 - \nabla f(\XX\ww_{-1})\|_1$, defining $\ww_{-1} \defas \ww_0$ for the sake of the proof;  $\rho = 1 - \mfrac{1}{n}$ and $K=\mfrac{2LD_1}{n}$. Specifying these values gives us the claimed bound.
\end{proof}

%% file: experiments_table.tex
\begin{table*}[ht]
\caption{Datasets and tasks used in experiments.}
\label{tab:datasets}
\vskip 0.15in
\begin{center}
\begin{small}
\begin{sc}
\begin{tabular}{lccccc}
\toprule
Dataset & $n$ & $d$ & $\kappa / n$ & $f_i$ & $\CC$ \\
\midrule
\href{https://archive.ics.uci.edu/ml/datasets/Breast+Cancer+Wisconsin+(Diagnostic)}{Breast Cancer} & 683 & 10 & 0.929 & $\log(1 + \exp(-\yy_i \xx_i^\top \ww))$ & $\{\|\ww\|_1\leq \lambda ,~\lambda = 5\}$  \\
\href{http://jmlr.csail.mit.edu/papers/volume5/lewis04a/}{RCV1} & 20,242 & 47,236 & 0.021 & $\log(1 + \exp(-\yy_i \xx_i^\top \ww))$ & $\{\|\ww\|_1\leq \lambda,~\lambda = 100\}$  \\
\href{http://lib.stat.cmu.edu/datasets/houses.zip}{California Housing} & 20,640 & 8 & 0.040 & $\tfrac{1}{2}(\yy_i - \xx_i^\top \ww)^2$ & $\{\|\ww\|_1 \le \lambda, \lambda=0.1\}$ \\
\bottomrule
\end{tabular}
\end{sc}
\end{small}
\end{center}
\vskip -0.2in
\end{table*}

%% file: experiments_batch_size1_plots.tex
\begin{figure*}[ht]
    \centering
    \includegraphics[width=\linewidth]{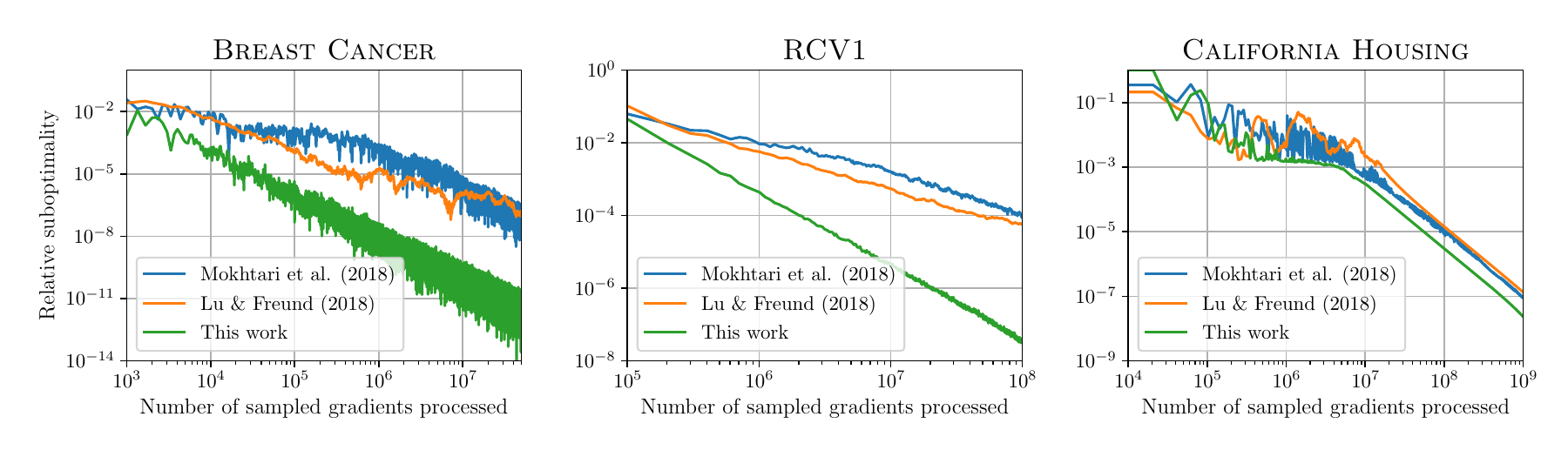}
    \caption{Comparing our SFW method to the related works of \citet{lu2018generalized} and \citet{mokhtari2018stochastic}. From left to right: \textsc{Breast Cancer}, \textsc{RCV1}, and \textsc{California Housing} datasets. We plot the relative subtimality values in $\log$-$\log$ plots to show empirical rates of convergence. We use the following batch size: $b=\floor{n/100}$.} 
    \label{fig:experiment-main}
\end{figure*}

%% file: appendix.tex
\appendix
\onecolumn

\section{Smoothness}
\label{apx:proof_smoothness}
\input{appendix_overall_smoothness.tex}
\section{Proof of Lemma \ref{lemma:sufficient_decrease}}
\label{apx:sufficient_decrease}
\input{appendix_proof_lemma1.tex}
\section{Completing the proof for Theorem~\ref{theorem:convex_rate}.}
\label{apx:tying_up}
\input{appendix_tying_up.tex}
\section{Bounds for $B_t$, $C_t$.}
\label{apx:taylor_bounds}
\input{appendix_taylor_bounds.tex}
\section{Proof of Proposition \ref{prop:proxy_gap_informal}.}
\label{apx:proof_proxy_gap}
\input{appendix_gap.tex}
\section{Proof of Theorem \ref{thm:nonconvex_rate}.}
\label{apx:proof_non_convex}
\input{appendix_proofs_nonconvex.tex}
\section{Comparison with other methods.}
\label{apx:comparison_other_methods}
\input{appendix_pseudocode_other_methods.tex}

%% file: appendix_overall_smoothness.tex
\begin{proposition}\label{prop:smoothness}
Let $f: \RR^n \to \RR$ be defined as $f(\btheta) = \frac{1}{n}\sum_i f_i(\btheta_i)$. If  $f_i$ is $L$-smooth for all $i\in \{1, \dots, n\}$, then $f$ satisfies \eqref{norman}  for every $p \in [1,\infty]$.
\end{proposition}

\begin{proof}

We observe that the $i$-th component of the gradient of $f$ is 

\begin{align}
    [\nabla f(\btheta)]_i = \frac{1}{n}f_i'(\btheta_i).
\end{align}

Recall that $|f'_i(\btheta_i) - f'_i(\bar\btheta_i)| \leq L|\btheta_i-\bar\btheta_i|$ for all $\btheta_i,~ \bar\btheta_i$ in the domain of $f_i$. Then, for the $\ell_p$ norm $\|\cdot\|_p$ and for all $\btheta,~\bar\btheta$ in the domain of $f$, the following holds

\begin{align}
    \|\nabla f(\btheta) - \nabla f(\bar\btheta)\|_p = \frac{1}{n}\sqrt[p]{\sum_{i=1}^n |f'_i(\btheta_i) - f'_i(\bar\btheta_i)|^p} \leq \frac{L}{n}\sqrt[p]{\sum_{i=1}^n |\btheta_i - \bar\btheta_i|^p} = \frac{L}{n} \|\btheta - \bar\btheta\|_p \ .
\end{align}

\end{proof}
\clearpage

%% file: appendix_proof_lemma1.tex
We adapt \cite{mokhtari2018stochastic}'s proof of Lemma \ref{lemma:sufficient_decrease}. For ease, we reproduce its statement first.

\begin{lemma}\label{evelyn}
Suppose $f$ is a convex function and is $(L/n)$-smooth with respect to the $\ell_2$ norm. For \emph{any} direction $\balpha\in \RR^n$, defining $\sss_t = \LMO(\XX\tran \balpha)$ and $\ww_t = (1-\gamma_t)\ww_{t-1} + \gamma_t \sss_t$, we have the following upper bound on the primal suboptimality
\begin{align}
\bsp
    \varepsilon_t  \leq~&  (1-\gamma_t)\varepsilon_{t-1} + \gamma_t^2\frac{LD^2_2}{2n}+ \gamma_t D_\infty H_t,
\esp
\end{align}
where $\varepsilon_t = f(\XX\ww_t) - f(\XX\ww_\star)$.
\end{lemma}
\begin{proof}

Recall the definition of $D_p =  \max_{\ww,~\vv \in \CC} \|\XX(\ww - \vv)\|_p$. 

\begin{align}
    f(\XX\ww_t) \leq& f(\XX\ww_{t-1}) + \langle \nabla f(\XX\ww_{t-1}), \XX(\ww_{t} - \ww_{t-1})\rangle + \frac{L}{2n} \|\XX(\ww_t - \ww_{t-1})\|_2^2 &\text{($(L/n)$-smoothness)} \\
    f(\XX\ww_t) \leq& f(\XX\ww_{t-1}) + \gamma_t\langle \nabla f(\XX\ww_{t-1}), \XX(\sss_{t} - \ww_{t-1})\rangle + \frac{\gamma_t^2 L}{2n} \|\XX(\sss_t - \ww_{t-1})\|_2^2 &\text{(Def of $\ww_t$)} \\
    \leq&  f(\XX\ww_{t-1}) + \gamma_t\langle \nabla f(\XX\ww_{t-1}), \XX(\sss_{t} - \ww_{t-1})\rangle + \gamma^2_t\frac{LD^2_2}{2n} &\text{(Def of $D_2$)} \\
    \bsp
    =& f(\XX\ww_{t-1}) + \gamma_t\langle \nabla f(\XX\ww_{t-1}) - \balpha, \XX(\sss_{t} - \ww_{t-1})\rangle \\
    & + \gamma_t\langle \balpha, \XX(\sss_t - \ww_{t-1})\rangle +  \gamma^2_t\frac{LD^2_2}{2n}
    \esp 
    & (\pm \gamma_t \langle \balpha, \XX(\sss_t - \ww_{t-1})\rangle) \\
    \bsp
    \leq& f(\XX\ww_{t-1}) + \gamma_t\langle \nabla f(\XX\ww_{t-1}) - \balpha, \XX(\sss_{t} - \ww_\star + \ww_\star - \ww_{t-1})\rangle \\
    &+ \gamma_t\langle \balpha, \XX(\ww_\star - \ww_{t-1})\rangle +  \gamma^2_t\frac{LD^2_2}{2n}
    \esp  
    \bsp
    &\text{($\pm \langle \nabla f(\XX \ww_{t-1}) - \balpha, \XX\ww_\star\rangle$)}\\
    &\text{(Opt. of $\sss_t$)}
    \esp \\
    \bsp
    =& f(\XX\ww_{t-1}) + \gamma_t \langle \nabla f(\XX\ww_{t-1}) - \balpha, \XX(\sss_t - \ww_{\star})\rangle  \\
    &+ \gamma_t\langle \nabla f(\XX\ww_{t-1}), \XX(\ww_\star - \ww_{t-1})\rangle  + \gamma_t^2\frac{LD^2_2}{2n}\esp &\text{(rewrite)} \\
    \bsp
    \leq& f(\XX\ww_{t-1})  + \gamma_t \langle \nabla f(\XX\ww_{t-1}), \XX(\ww_{\star} - \ww_{t-1})\rangle  \\
    &+ \gamma_t D_\infty \|\nabla f(\XX\ww_{t-1}) - \balpha\|_1 + \gamma_t^2\frac{LD^2_2}{2n}\esp &\text{(Hölder's inequality and def of $D_\infty$)} \\
    \bsp
    \leq& f(\XX\ww_{t-1}) + \gamma_t (f(\XX\ww_\star) - f(\XX\ww_{t-1})) + \gamma_t    D_\infty \|\nabla f(\XX\ww_{t-1}) - \balpha\|_1  \\
    &+ \gamma_t^2\frac{LD^2_2}{2n}\esp &\text{(Convexity of $f$)}
    \end{align}
    Subtracting $f(\XX\ww_\star)$ on both sides, we get
    \begin{align}
    f(\XX\ww_{t}) - f(\XX\ww_{\star}) \leq (1-\gamma_t)(f(\XX\ww_{t-1}) - f(\XX\ww_{\star})) + \gamma_t  D_\infty \|\nabla f(\XX\ww_{t-1}) - \balpha\|_1 + \gamma_t^2\frac{LD^2_2}{2n}.
\end{align}
We define $H_t = \|\balpha - \nabla f(\XX\ww_{t-1})\|_1$, and recall the definition of $\varepsilon_t$ to obtain the claimed bound
\begin{align}
    \varepsilon_t \leq (1-\gamma_t) \varepsilon_{t-1} + \gamma_t^2\frac{LD^2_2}{2n} + \gamma_t D_\infty H_t.
\end{align}
\end{proof} 

\pagebreak

%% file: appendix_tying_up.tex
Given the three key Lemmas 1-3, we can finish the proof. Under the hypotheses of Theorem~\ref{theorem:convex_rate}, let us consider step $t$ of the SFW algorithm.

We plug our upper bound on $H_t$ \eqref{rainy} into the upper bound from Lemma 1 \eqref{eq:subopt_last} and take expectations on both sides to obtain the following upper bound on the expected primal-suboptimality~$\EE\varepsilon_t$.

\begin{align}
\bsp
    \EE \varepsilon_t \leq& (1-\gamma_t)\EE\varepsilon_{t-1} + \gamma_t^2 \frac{LD^2_2}{2n} \\
    &+ \gamma_t\mfrac{2 L D_1 D_\infty}{n} \left(\frac{2(n-1)}{t+2} + \left(1-\frac{1}{n}\right)^{t/2}\log t\right)\\
    &+ \gamma_t D_\infty \left(1-\frac{1}{n}\right)^t H_0 .
\esp
\end{align}

By specifying the step-size $\gamma_t = \frac{2}{t+2}$ and multiplying the previous inequality by $(t+1)(t+2)$, we get an expression in which the expected sub-optimalities telescope under summation. This allows us to get the promised rate. For simplicity, we upper bound $\frac{t+1}{t+2}$ by $1$.

Let $\Gamma_t = (t+1)(t+2)\EE\varepsilon_t $. We have

\begin{align}
\bsp
    &\Gamma_t \leq \Gamma_{t-1} + 2 \mfrac{LD^2_2}{n}+ 8 \mfrac{(n-1)}{n}{L D_1 D_\infty} \\
    &\quad+ 4 \mfrac{LD_1 D_\infty}{n} (t+1)\!\!\left(1-\mfrac{1}{n}\right)^{t/2}\log t\\
    &\quad+ 2 D_\infty H_0 (t+1) \left(1-\mfrac{1}{n}\right)^t
\esp
\end{align}

If we sum this expression over time-steps $k=1,\dots, t$, we obtain 

\begin{align}
\bsp
    \Gamma_t \leq& \Gamma_0 + 2L\left( \frac{D^2_2 + 4(n-1)D_1D_\infty}{n}  \right) t  \\
    &\quad+ 4 \mfrac{L D_1 D_\infty}{n} B_t\\
    &\quad+ 2 D_\infty H_0 C_t\,,
\esp
\end{align}

where 

\begin{align}
    B_t =& \sum_{k=1}^t(k+1)\left(1-\frac{1}{n}\right)^{k/2}\log k \leq 16 n^3 \\
    C_t =&  \sum_{k=1}^t(k+1)\left(1-\frac{1}{n}\right)^k \leq n^2\,.
\end{align}
These bounds use Taylor series and are proven in Appendix  \ref{apx:taylor_bounds}. By combining the previous two bounds we get the following upper bound
\begin{align}
\bsp
    \Gamma_t &\leq \Gamma_0 + 2L\left( \frac{D^2_2 + 4(n-1)D_1D_\infty}{n}  \right) t  \\
    &\quad+ (2 D_\infty H_0 + 64 {L D_1 D_\infty}) n^2.
\esp
\end{align}

We divide this upper bound by $(t+1)(t+2)$, and finally use the bound $\frac{1}{(t+1)(t+2)}\leq \frac{1}{t^2}$ to obtain the following rate on $\EE\varepsilon_t$:

\begin{align}
    \EE\varepsilon_t &\leq 2L\left( \frac{D^2_2 + 4(n-1)D_1D_\infty}{n}  \right) \frac{t}{(t+1)(t+2)} +  \frac{2 \varepsilon_0 + (2 D_\infty H_0 + 64 {L D_1 D_\infty}) n^2}{(t+1)(t+2)} \ . 
\end{align}

\pagebreak

%% file: appendix_taylor_bounds.tex
For $B_t$, we use the (aggressive) bound $\log k \leq k-1$ and notice that $\sum_{k=1}^{\infty} (k+2)(k+1)\rho^{k} = \frac{2}{(1-\rho)^3}$ to get 

\begin{align}
    B_t &\leq \sum_{k=1}^t(k-1)(k+1)\left(1-\frac{1}{n}\right)^{k/2} \\
    &\leq \sum_{k=1}^t (k+2)(k+1)\left(1-\frac{1}{n}\right)^{k/2} \\
    &\leq 2\left(\frac{1}{1- \sqrt{1-\frac{1}{n}}}\right)^3 \\
    &= 2n^3\left(1 + \sqrt{1-\frac{1}{n}}\right)^3 \leq 16n^3.
\end{align}

Notice that $C_t$ is the beginning of the Taylor series expansion of $\frac{d}{dx}\frac{1}{1-x} = \frac{1}{(1-x)^2}$, for $x=\frac{n-1}{n}$. We can upper bound it by the full series, leading to

\begin{align}
    C_t \leq& \left(\frac{1}{1-\left(1-\frac{1}{n}\right)}\right)^2 = n^2.
\end{align}
\clearpage

%% file: appendix_gap.tex
\begin{proof} It suffices to prove that 

\begin{align}
    |g_t - \hat{g}_t | \leq D_\infty H_t \ . 
\end{align}

We recall the definitions of the true and stochastic FW gaps:
\begin{align}
    g_t &= \max_{\sss \in \CC} \langle \nabla f(\XX\ww_{t-1}), \XX(\ww_{t -1} - \sss)\rangle \defas \langle \nabla f(\XX\ww_{t-1}), \XX(\ww_{t -1} - \sss_t)\rangle \\
    \hat g_t &= \max_{\sss \in \CC} \langle \balpha_t, \XX(\ww_{t -1} - \sss)\rangle \defas \langle \balpha_t, \XX(\ww_{t -1} - \hat \sss_t)\rangle
\end{align}

where we associate $\sss_t$ to the true gap, and $\hat \sss_t$ to the stochastic gap.

Now, 

\begin{align}
    g_t &= \langle \nabla f (\XX\ww_{t-1}), \XX\left(\ww_{t-1} - \sss_t\right)\rangle \\
    &= \langle \balpha_t, \XX\left(\ww_{t-1} - \sss_t\right)\rangle + \langle \nabla f (\XX\ww_{t-1}) - \balpha_t, \XX\left(\ww_{t-1} - \sss_t\right)\rangle\\
    &\leq \langle \balpha_t, \XX\left(\ww_{t-1} - \hat \sss_t\right)\rangle + \langle \nabla f (\XX\ww_{t-1}) - \balpha_t, \XX\left(\ww_{t-1} - \sss_t\right)\rangle\\
    &\leq \hat g_t + D_\infty H_t,
\end{align}
where the first inequality results from optimality of $\hat \sss_t$, and the second inequality results from Hölder's inequality and the definitions of $H_t$ and $D_\infty$.

Both gaps $g_t$ and $\hat g_t$ play symmetric roles in the previous bounds, therefore, we also have the bound: 
\begin{align}
    \hat g_t \leq g_t + D_\infty H_t,
\end{align}
thus concluding the proof.
\end{proof}

\pagebreak

%% file: appendix_proofs_nonconvex.tex
Let us now show that when the $f_i$s are $L$-smooth, and the iterates are given by the proposed SFW, then $\liminf_{t\to\infty} \EE_t[g_t] = 0$.

\begin{proof}
We adapt the proof of Lemma~\ref{lemma:sufficient_decrease}. At step $t$, using Proposition~\ref{prop:smoothness} with $p=2$, we obtain
\begin{align}
    f(\XX\ww_t) &\leq  f(\XX\ww_{t-1}) + \gamma_t \langle \nabla f(\XX\ww_{t-1}), \XX(\sss_t - \ww_{t-1})\rangle + \gamma_t^2 \frac{L D_2^2}{2n} \\
    & = f(\XX\ww_{t-1}) - \gamma_t \hat g_t + \gamma_t\langle \nabla f(\XX\ww_{t-1}) - \balpha_t, \XX(\sss_t - \ww_{t-1})\rangle + \gamma_t^2 \frac{L D^2_2}{2n} \\
    &\leq f(\XX\ww_{t-1}) - \gamma_t \hat g_t + \gamma_t D_\infty H_t + \gamma_t^2 \frac{L D^2_2}{2n} \ .
\end{align}

Rearranging, we have 

\begin{align}
    \gamma_t\hat{g}_t \leq f(\XX\ww_{t-1}) - f(\XX\ww_{t})+ \gamma_t D_\infty H_t + \gamma_t^2 \frac{L D_2^2}{2n}.
\end{align}

Therefore, summing for $u=1,\dots, t$ 

\begin{align}
    \sum_{u=1}^t \gamma_u\hat{g}_u \leq f(\XX\ww_0) -  f(\XX\ww_{t})+ \sum_{u=1}^t \gamma_u D_\infty H_u + \gamma_u^2 \frac{L D_2^2}{2n}.
\end{align}

The right hand side is bounded in expectation: $f$ is continuous on the compact set $\CC$, and the series converges, since $\EE_t H_t=\mathcal{O}(\frac{1}{t})$ and $\gamma_t^2 = \mathcal{O}(1/t^2)$. This implies that $\liminf \EE_t \hat{g}_t = 0$, since $\gamma_t = \frac{2}{t+2}$ is not the general term of a convergent series. Finally, since $|g_t - \hat{g}_t| \leq D_\infty H_t$ (Appendix \ref{apx:proof_proxy_gap}), this yields the claimed result.
\end{proof}
\clearpage

%% file: appendix_pseudocode_other_methods.tex
To make the comparison with other methods easier to grasp and to implement for the interested reader, we report pseudo code using our notation for the Stochastic Frank-Wolfe algorithms in \citet{lu2018generalized} and \citet{mokhtari2018stochastic}. In the case of \citet{mokhtari2018stochastic}, we also specify their algorithm in the same formal setting as ours where $f(\XX\ww) = \frac{1}{n}f_i(\xx_i\tran\ww)$ and the sampling is over datapoints.

\subsection{\citet{mokhtari2018stochastic}}

\citet{mokhtari2018stochastic} have two sets of step-sizes, which we denote by $\rho_t$, $\gamma_t$. They use a form of momentum on an \emph{unbiased} estimator of the gradient using the $\rho_t$ step sizes. The values they use are $\gamma_t = \frac{1}{t+1}$ and $\rho_t = \frac{1}{(t+1)^{2/3}}$.

\begin{algorithm}[hb]
  \caption{\citet{mokhtari2018stochastic}}
  \label{alg:mok}
\begin{algorithmic}[1]
  \STATE {\bfseries Initialization:} $\ww_0\in\CC$, $\balpha_0=0$, $\rr_0 =0$
    \FOR{$t=1, 2, \dots, $}

        \STATE Sample $i$ uniformly at random in $\{1, \dots, n\}$ \\
        \STATE $\balpha_t^i = (1-\rho_t)\balpha_{t-1}^i + \rho_t f_i'(\xx_i\tran \ww_{t-1})$ \\
        \STATE $\rr_t = \rr_{t-1} + (\balpha_t^i - \balpha_{t-1}^i)\xx_i$
        \STATE $\sss_{t} = \LMO(\rr_{t})$\\
        \STATE $\ww_{t} = (1 - \gamma_t)\ww_{t-1} + \gamma_t \sss_t$ 
  \ENDFOR
\end{algorithmic}
\end{algorithm}

\subsection{\citet{lu2018generalized}}

\citet{lu2018generalized} also have two step-size sequences given by $\gamma_t = \frac{2(2n_b+t)}{(t+1)(4n_b+t+1)}$ and $\delta_t = \frac{2n_b}{2n_b + t + 1}$, where $n_b$ is the number of batches, i.e. $\floor{n/b}$, with $n$ the number of samples in the dataset, and $b$ the chosen batch size. They use a form of momentum on the argument to a given $f_i$, and compute the gradient at an averaged iterate, which we denote by $\bsigma^i_t$.
In our notation, $t$ is the iteration step and $i$ corresponds to the $i$-th datapoint.

\begin{algorithm}[hb]
  \caption{\citet{lu2018generalized}}
  \label{alg:lu_f}
\begin{algorithmic}[1]
  \STATE {\bfseries Initialization:} $\ww_0\in\CC$, $\bsigma_0=\XX\ww_0$, $\balpha_0=0$, $\rr_0 =0$
    \FOR{$t=1, 2, \dots, $}
        \STATE $\sss_{t} = \LMO(\rr_{t-1})$
        
        \STATE Sample $i$ uniformly at random in $\{1, \dots, n\}$ 
        
        \STATE $\bsigma_t^i = (1-\delta_t) \bsigma^i_{t-1} + \delta_t( \xx_i\tran\sss_{t})$ 
        
         \STATE $\balpha_t^i = \frac{1}{n}f'_i(\bsigma_t^i)$ 
         
        \STATE $\rr_t = \rr_{t-1} +\left(\balpha_t^i - \balpha_{t-1}^i\right) \xx_i $
        
        \STATE $\ww_{t} = (1 - \gamma_t)\ww_{t-1} + \gamma_t \sss_t$ 
  \ENDFOR
\end{algorithmic}
\end{algorithm}

\pagebreak